\documentclass[12pt]{amsart}
\usepackage{fullpage,url,mathrsfs,setspace,amssymb}
\usepackage[all]{xy} 

\DeclareFontEncoding{OT2}{}{} 

\usepackage{color}

\newcommand{\Strut}{\rule{0pt}{13pt}}   
\newcommand{\genr}[1]{\langle#1\rangle} 

\newcommand{\defi}[1]{\textsf{#1}} 

\newcommand{\Aff}{{\mathbb A}}

\newcommand{\G}{{\mathbb G}}

\newcommand{\PP}{{\mathbb P}}
\newcommand{\Q}{{\mathbb Q}}

\newcommand{\Z}{{\mathbb Z}}
\newcommand{\Qbar}{{\overline{\Q}}}

\def\bbar#1{\setbox0=\hbox{$#1$}\dimen0=.2\ht0 \kern\dimen0 
\overline{\kern-\dimen0 #1}}
 
\newcommand{\kbar}{{\bbar{k}}}
\newcommand{\Xbar}{X_{\kbar}}

\newcommand{\pp}{{\mathfrak p}}

\newcommand{\s}{{\mathfrak s}}

\newcommand{\N}{{\bbar{N}_{L/k}}}

\newcommand{\sF}{{\mathscr F}}

\newcommand{\calQ}{{\mathcal Q}}

\newcommand{\calX}{{\mathcal X}}

\newcommand{\OO}{{\mathscr O}}
\newcommand{\sA}{{\mathscr A}}

\DeclareMathOperator{\Char}{char}
\DeclareMathOperator{\inv}{inv}

\DeclareMathOperator{\im}{im}

\DeclareMathOperator{\Aut}{Aut}
\DeclareMathOperator{\Gal}{Gal}

\DeclareMathOperator{\Br}{Br}

\DeclareMathOperator{\Div}{Div}
\DeclareMathOperator{\Pic}{Pic}

\DeclareMathOperator{\Spec}{Spec}

\DeclareMathOperator{\Proj}{Proj}

\DeclareMathOperator{\cyc}{cyc}

\newtheorem{theorem}{Theorem}[section]
\newtheorem{lemma}[theorem]{Lemma}

\newtheorem{proposition}[theorem]{Proposition}

\theoremstyle{definition}
\newtheorem{definition}[theorem]{Definition}

\theoremstyle{remark}
\newtheorem{remark}[theorem]{Remark}

\usepackage[
	colorlinks,
	backref,
	pdfauthor={Anthony V\'arilly-Alvarado}, 
]{hyperref}
\usepackage[alphabetic,backrefs,lite]{amsrefs} 

\begin{document}

\title[short paper title]{Weak approximation on del Pezzo surfaces of degree 1}
\author{Anthony V\'arilly-Alvarado}
\thanks{This research was partially supported by a Marie Curie Research Training Network within the 6th European Community Framework Program.}
\address{Department of Mathematics, University of California, 
	Berkeley, CA 94720-3840, USA}
\email{varilly@math.berkeley.edu}
\urladdr{http://math.berkeley.edu/\~{}varilly}
\date{January 8th, 2009}

\begin{abstract}
We study del Pezzo surfaces of degree $1$ of the form 
\[ 
w^2 = z^3 + Ax^6 + By^6 
\] 
in the weighted projective space $\PP_k(1,1,2,3)$, where $k$ is a perfect field of characteristic not $2$ or $3$ and $A,B \in k^*$. Over a number field, we exhibit an infinite family of (minimal) counterexamples to weak approximation amongst these surfaces, via a Brauer-Manin obstruction.
\end{abstract}

\maketitle

\section{Introduction}\label{S:introduction}

Let $X$ be a geometrically integral variety over a number field $k$. Write $\Omega_k$ for the set of places of $k$, and let $k_v$ be the completion of $k$ at $v \in \Omega_k$. Assume that $X$ has $k_v$-points at every place $v$. We say that $X$ satisfies \defi{weak approximation} if the diagonal embedding
\[
X(k) \hookrightarrow \prod_{v\in \Omega_k} X(k_v)
\]
is dense for the product of the $v$-adic topologies. If $X'$ is another $k$-variety, $k$-birational to $X$, and both $X$ and $X'$ are smooth, then $X'$ satisfies weak approximation if and only if $X$ does. As Swinnerton-Dyer puts it, the ``dramatic'' failure of weak approximation, that is, when $X(k) = \emptyset$ and yet $X(k_v)\neq \emptyset$ for every place $v$, is referred to as the failure of the \defi{Hasse principle}; see~\cite{Harari}.

A \defi{del Pezzo surface} $X$ is a smooth projective geometrically rational surface with ample anticanonical class $-K_X$. The \defi{degree} $d$ of $X$ is $K_X^2$; it is an integer in the range $1\leq d \leq 9$. When $d \geq 5$, $X$ is known to satisfy both the Hasse principle and weak approximation.  On the other hand, there are counterexamples to both of these phenomena when $d$ is $2$, $3$ and $4$, all of which can be explained by a Brauer-Manin obstruction; see~\cite{KreschTschinkel},~\cite{CasselsGuy} and~\cite{BirchSwinnertonDyer}, respectively.  Del Pezzo surfaces of degree $1$ satisfy the Hasse principle because they come furnished with a rational point: the base-point of the anticanonical linear system.  (We refer to this point as the \defi{anticanonical point}.)  We will give examples of these surfaces that do not satisfy weak approximation, following ideas of Kresch and Tschinkel~\cite{KreschTschinkel}.

Let $k$ be a perfect field and let $k[x,y,z,w]$ be the weighted graded ring where the variables $x,y,z,w$ have weights $1,1,2,3$, respectively. Set $\PP_k(1,1,2,3) := \Proj k[x,y,z,w]$. Every del Pezzo surface of degree $1$ over $k$ is isomorphic to a smooth sextic hypersurface in $\PP_k(1,1,2,3)$. Conversely, any smooth sextic in $\PP_k(1,1,2,3)$ is a del Pezzo surface of degree $1$ over $k$ (see \S\ref{S:anticanonical models}). Our main result is as follows.

\begin{theorem}
\label{T:maintheorem}
Let $p\geq 5$ be a rational prime number such that $p \not\equiv 1 \bmod 12$. Let $X$ be the del Pezzo surface of degree $1$ over $\Q$ given by
\[
w^2 = z^3 + p^3x^6 + p^3y^6
\] 
in $\PP_\Q(1,1,2,3)$.  Then $X$ is $\Q$-minimal and there is a Brauer-Manin obstruction to weak approximation on $X$. Moreover, the obstruction arises from a cyclic algebra class in $\Br X/\Br \Q$.
\end{theorem}

To obtain these examples, we begin with an explicit study of the geometry of \emph{diagonal} del Pezzo surfaces of degree $1$ over a perfect field $k$ with $\Char k \neq 2, 3$. These are sextic surfaces of the form
\begin{equation}
w^2 = z^3 + Ax^6 + By^6
\label{E:diagonalsurfaces}
\end{equation}
in the weighted projective space $\PP_k(1,1,2,3)$, where $A, B \in k^*$. The conditions $A$, $B \in k^*$ and $\Char k \neq 2, 3$, taken together, are equivalent to the smoothness of these surfaces.

Let $R$ be a graded ring and let $I \subseteq R$ be a homogeneous ideal. Then $V(I) := \Proj R/I$. If $I = (f_1,\cdots f_n)$ we write $V(f_1,\dots,f_n)$ instead of $V((f_1,\dots,f_n))$. We start by finding an explicit description of generators for the geometric Picard group for the surfaces~(\ref{E:diagonalsurfaces}).  More generally, we find explicit equations for all $240$ exceptional curves on a del Pezzo surface of degree $1$ over any perfect field. 

\begin{theorem}  \label{T:lines}
Let $X$ be a del Pezzo surface of degree 1 over a perfect field $k$, given as a smooth sextic hypersurface $V(f(x,y,z,w))$ in $\PP_k(1,1,2,3)$. Let 
\[
\Gamma = V(z - Q(x,y), w - C(x,y)) \subseteq \PP_{\kbar}(1,1,2,3),
\]
where $Q(x,y)$ and $C(x,y)$ are homogenous forms of degrees $2$ and $3$, respectively, in $\kbar[x,y]$. If $\Gamma$ is a divisor on $X_\kbar := X \times_k \kbar$, then it is an exceptional curve of $X$. Conversely, every exceptional curve on $X$ is a divisor of this form.
\end{theorem}

With explicit generators for $\Pic\Xbar$ for a surface $X$ of the form~(\ref{E:diagonalsurfaces}), we may compute the cohomology group $H^1(\Gal(\kbar/k),\Pic X_\kbar)$.  We derive the following theorem, analogous to~\cite[Thm.\ 1]{KreschTschinkel}.

\begin{theorem}
\label{T:possibleH1s}
Let $k$ be a perfect field with $\Char k \neq 2,3$. Let $X$ be a minimal del Pezzo surface of degree $1$ over $k$ of the form~\emph{(\ref{E:diagonalsurfaces})}. Then $H^1(\Gal(\kbar/k),\Pic(X_\kbar))$ is isomorphic to one of the following fourteen groups:
\begin{gather*}
\{1\}; \quad (\Z/2\Z)^i,\ \ i \in \{1,2,3,4,6,8\};  \quad (\Z/3\Z)^j,\ \ j \in \{1,2,3,4\}; \\
(\Z/6\Z)^k\ \ k\in \{1,2\}; \quad \Z/2\Z\times \Z/6\Z.
\end{gather*}
Each group occurs for some field $k$. When $k = \Q$ only the following seven groups occur:
\begin{gather*}
\{1\},\quad \Z/2\Z ,\quad \Z/2\Z\times \Z/2\Z, \quad \Z/2\Z\times \Z/2\Z\times \Z/2\Z, \\
\quad \Z/3\Z, \quad \Z/3\Z\times \Z/3\Z,  \quad \Z/6\Z.
\end{gather*}
\end{theorem}

If, furthermore, $k$ is a number field, then we may compute the group $\Br X /\Br k$, of arithmetic interest, via the isomorphism
\begin{equation}
\Br X/\Br k \xrightarrow{\sim} H^1(\Gal(\kbar/k),\Pic X_\kbar),
\label{E:HochshildSerreIso}
\end{equation}
obtained from the Hochschild-Serre spectral sequence (see, for example,~\cite[Lemme 15]{Requivalence}).  

To prove a statement like Theorem~\ref{T:maintheorem}, we have to identify elements of $\Br X/\Br k$ explicitly. Given a cohomology class in $H^1(\Gal(\kbar/k),\Pic X_\kbar)$, it can be difficult to identify the corresponding element in $\Br X/\Br k$ guaranteed by the isomorphism~(\ref{E:HochshildSerreIso}). Hence, in \S\ref{S:cyclic} we present a simple strategy to search for cohomology classes in $H^1(\Gal(\kbar/k),\Pic X_\kbar)$ which correspond to cyclic algebras in the image of the natural map 
\[
\Br X/\Br k \to \Br k(X)/\Br k,
\] 
where $X$ is a locally soluble smooth geometrically integral variety over a number field $k$. We hope that Theorem~\ref{T:cyclicalgebras} will be of use to others wishing to calculate Brauer-Manin obstructions to the Hasse principle and weak approximation via cyclic algebras on this wide class of varieties.

The paper is organized as follows. In \S\ref{S:background} we review a few basic facts about del Pezzo surfaces and Brauer-Manin obstructions.  In \S\ref{S:cyclic} we present a strategy to search for cyclic algebras in the image of the natural map $\Br X/\Br k \to \Br k(X)/\Br k$, as explained above. In \S\ref{S:lines} we prove Theorem~\ref{T:lines} and in \S\ref{S:CurvesOnDiagonals} we use it to write down generators for the geometric Picard group on a surface $X$ of the form~(\ref{E:diagonalsurfaces}). In \S\ref{S:galois} we compute the action of $\Gal(\kbar/k)$ on $\Pic\Xbar$ and the possible groups $H^1(\Gal(\kbar/k),\Pic(X_\kbar))$. Finally we prove Theorem~\ref{T:maintheorem} in \S\ref{S:counterexamples}.

\subsection{Notation}\label{S:notation}

In addition to the notation introduced above, we use the following conventions.
Throughout $k$ denotes a perfect field and $\kbar$ is a fixed algebraic closure of $k$. From \S\ref{S:CurvesOnDiagonals} onwards we assume $\Char k \neq 2, 3$; in this case $A$ and $B$ denote elements of $k^*$; $\alpha$ and $\beta$ are fixed sixth roots of $A$ and $B$, respectively, in $\kbar$. Also, $\zeta$ denotes a primitive sixth root of unity in $\kbar$ and $s$ a fixed cube root of $2$ in $\kbar$.

If $X$ and $Y$ are $S$-schemes then $X_Y := X\times_S Y$. If $Y = \Spec R$ then we write $X_R$ instead of $X_{\Spec R}$. For an integral scheme $X$ over a field we write $k(X)$ for the function field of $X$. A \defi{surface} $X$ is a separated integral scheme of finite type over a field $k$ of dimension $2$. If $X$ is a locally factorial projective surface, then there is an intersection pairing on the Picard group 
$
(\,\cdot\, ,\cdot\, )_X \colon \Pic X \times \Pic X \to \Z.
$
We omit the subscript on the pairing if no confusion can arise. For such an $X$, we will identify $\Pic(X)$ with the Weil divisor class group; in particular, we will use additive notation for the group law on $\Pic(X)$.

For a smooth variety $X$ over a number field $k$, and a Galois extension $L/k$, we write $ N_{L/k}\colon \Div X_L \to \Div X_k$ and $\N\colon \Pic X_L \to \Pic X_k$ for the usual norm maps, respectively.

\subsection*{Acknowledgements}

I thank my advisor Bjorn Poonen for countless useful conversations. I also thank Andrew Kresch, Patrick Corn, Dan Erman and Mark Haiman for many useful discussions.  I thank the referees for valuable comments and suggestions.  All computations were done using {\tt Magma}~\cite{magma}. The relevant scripts are available from the author.

\section{Background}\label{S:background}

We begin by reviewing some well known facts about del Pezzo surfaces over a field $k$. The basic references on the subject are~\cite{Manin}, ~\cite{Demazure1980} and~\cite[III.3]{Kollar1996}.  Unless otherwise stated, $X$ denotes a del Pezzo surface over a field $k$ of degree $d$ such that $X_{\kbar} \ncong \PP^1_{\kbar}\times \PP^1_{\kbar}$.

\subsection{Picard groups}
\label{S:blowup}
Recall an exceptional curve on $X$ is an irreducible curve $C$ on $\Xbar$ such that $(C,C) = (C,K_X) = -1$. When $d = 1$, $X$ contains $240$ exceptional curves. The group $\Pic X_\kbar$ is isomorphic to $\Z^{10-d}$; it is generated by the classes of exceptional curves. A possible basis for $\Pic \Xbar$ is $\{e_1,\dots,e_{9-d},\ell\}$, where each $e_i$ is the class of an exceptional curve, and
\[
(e_i,e_j) = -\delta_{ij},\quad 
(e_i,\ell) = 0,\quad
(\ell,\ell) = 1.
\]
Under this basis, the anticanonical class is given by $-K_X = 3\ell - \sum e_i$.

\subsection{Anticanonical models}
\label{S:anticanonical models}

If $X$ is a del Pezzo surface then $X \cong \Proj \bigoplus_{m\geq 0} H^0(X,-mK_X)$ \cite[Theorem III.3.5]{Kollar1996}. The latter scheme is known as the \defi{anticanonical model} of $X$. When $d=1$ the anticanonical model is a smooth sextic hypersurface $V(f(x,y,z,w))$ in $\PP_k(1,1,2,3)$. Any smooth sextic hypersurface in $\PP_k(1,1,2,3)$ is a del Pezzo surface of degree $1$. In this case $\{x,y\}$ is a basis for $H^0(X,-K_X)$ and $\{x^2,xy,y^2,z\}$ is a basis for $H^0(X,-2K_X)$.

\subsection{The Bertini involution} \label{SS:Bertini}
Let $X$ be a del Pezzo surface of degree $1$ given as a smooth sextic $V(f)$ in $\PP_k(1,1,2,3)$. Write $f(x,y,z,w) = w^2 - aw - b$, where $a, b \in k[x,y,z]$ have degrees $3$ and $6$, respectively. If $\Char k \neq 2$, then we may (and do) assume that $a \equiv 0$ by making the change of variables $w\mapsto w + a/2$. The map
\[
\PP_k(1,1,2,3) \to \PP_k(1,1,2,3),\quad [x:y:z:w] \mapsto [x:y:z:-w + a]
\]
restricts to an automorphism of $X$ called the \defi{Bertini involution}; see~\cite[p. 68]{Demazure1980}. 

\subsection{Galois action on the Picard Group}
\label{S:GaloisAction}

In this section $X$ is a smooth, projective, geometrically rational surface over a number field $k$.
Let $K$ be the smallest subfield of $\kbar$ over which all exceptional curves of $X$ are defined. We say $K$ is the \defi{splitting field} of $X$. The natural action of $\Gal(\kbar/k)$ on $\Pic X_\kbar \cong \Pic X_K$ factors through the quotient $\Gal(K/k)$, giving a map
\begin{equation}
\phi_X\colon \Gal(K/k) \to \Aut(\Pic X_K).
\label{E:FiniteGaloisRep}
\end{equation}
If we have equations for an exceptional curve $C$ of $X$, then an element $\sigma \in \Gal(K/k)$ acts on $C$ by applying $\sigma$ to each coefficient. The curve ${}^\sigma\! C$ is itself an exceptional curve of $X$. 

If, furthermore, $X$ is a del Pezzo surface of degree $1$. then the image of $\phi_X$ is isomorphic to a subgroup of the Weyl group $W(E_8)$ (which is a finite group of order $696, 729, 600$). To keep computations reasonable when searching for counterexamples to weak approximation, we work with surfaces $X$ for which $\im\phi_X$ is small. On the other hand, the image cannot be too small: for example, if $\im\phi_X = \{1\}$, then $X$ is $k$-birational to $\PP^2_k$, so it satisfies weak approximation.

\subsection{Minimal surfaces}
There are examples of del Pezzo surfaces of degrees $2$, $3$ and $4$ with a Zariski dense set of rational points for which weak approximation does not hold (cf.~\cite{KreschTschinkel2},~\cite{SD} and~\cite[15.5]{CTSSD}, respectively). These examples can be used to construct nonminimal del Pezzo surfaces of degree $1$ that do not satisfy weak approximation.  To avoid such examples, we will insist that our surfaces be $k$-minimal.

\begin{definition}
We say $X$ is \defi{$k$-minimal} if there is no $\Gal(\kbar/k)$-stable set $S$ of exceptional curves such that $(s_i,s_j) = -\delta_{ij}$ for every $s_i, s_j \in S$.
\end{definition}

Del Pezzo surfaces $X$ with $\Pic X\cong \Z$ are minimal. The converse is true if $d\notin \{1, 2,4\}$; see~\cite[Rem.\ 28.1.1]{Manin}.

\subsection{Brauer-Manin obstructions}
\label{S:BMobstructions}
We refer the reader to~\cite{Skorobogatov} for a thorough treatment of the material in this section.
If $X$ is a smooth projective geometrically integral variety over a number field $k$, then the natural inclusion
$
 X(\Aff_k) \subseteq \prod_{v \in \Omega_k} X(k_v)
$
is a bijection. Let $\Br X$ be the group of equivalence classes of Azumaya algebras on $X$, and let $\inv_v\colon\Br k_v \to \Q/\Z$ be the local invariant map. By class field theory there is a constraint
\begin{equation*}
X(k) \subseteq X(\Aff_k)^{\Br} := \Big\{(x_v)_v \in X(\Aff_k) \, |\,  \sum_v \inv_v(\sA(x_v))  = 0 \text{ for every }\sA \in \Br X \Big\},
\end{equation*}
where $\sA(x_v) := \sA_{x_v}\otimes_{\OO_{X,x_v}}k_v$. In fact, the closure $\overline{X(k)}$ of $X(k)$ in $X(\Aff_k)$ lies inside the set $X(\Aff_k)^{\Br}$. We say there is a \defi{Brauer-Manin obstruction to the Hasse principle} (resp.\ \defi{weak approximation}) if $X(\Aff_k)^{\Br} = \emptyset$ but $X(\Aff_k)\neq \emptyset$ (resp. if  $X(\Aff_k)^{\Br} \neq X(\Aff_k)$). We remark that to compute $X(\Aff_k)^{\Br}$ it suffices to consider a set of representatives of $\Br X/\Br k$. We also note that when $X(\Aff_k)\neq \emptyset$ the natural map $\Br k \to \Br X$ is an injection.

For $X$ as above, we have $\Br X \cong H^2_{\text{\'et}}(X,\G_m)$. This allows us to think of the Brauer group as a contravariant functor with values in the category of abelian groups.

\section{Finding cyclic algebras in $\Br X$}\label{S:cyclic}

When $X$ is a regular, integral, quasi-compact scheme the natural map $\Br X \to \Br k(X)$ is injective (see~\cite[III.2.22]{Milne}). There are certain elements of $\Br k(X)$ whose local invariants are easy to compute.  They are the \defi{cyclic algebras}.

\subsection{Review of cyclic algebras}
\label{S:cyclicalgebras}
Let $L/k$ be a finite cyclic extension of fields of degree $n$. Fix a generator $\sigma$ of $\Gal(L/k)$.
Let $L[x]_\sigma$ be the ``twisted'' polynomial ring, where $\ell x = x{}^\sigma\!\ell$ for all $\ell \in L$.
Given $b \in k^*$ we construct the central simple $k$-algebra $L[x]_\sigma/(x^n - b)$. This algebra is usually denoted $(L/k,b)$: it depends on the choice of $\sigma$, though the notation does not show this.

If $X$ is a geometrically integral $k$-variety, then the cyclic algebra $(k(X_L)/k(X), f)$ is also denoted $(L/k,f)$; this should not cause confusion because $\Gal(k(X_L)/k(X)) \cong \Gal(L/k)$.

The following is a criterion for testing whether or not a cyclic algebra is in the image of the map $\Br X \to \Br k(X)$. For a proof, see~\cite[Prop. 2.2.3]{Corn} or~\cite[Prop. 4.17]{Bright}. See \S\ref{S:notation} for our conventions on the norm maps $N_{L/k}$ and $\N$.

\begin{proposition}
\label{P:normsarecyclic}
Let $X$ be a smooth, geometrically integral variety over a number field $k$. Let $L/k$ a finite cyclic extension and $f \in k(X)^*$. Then the cyclic algebra $(L/k,f)$ is in the image of the natural map $\Br(X) \to \Br k(X)$ if and only if $(f) = N_{L/k}(D)$, for some $D \in \Div X_L$. If $X(k_v) \neq \emptyset$ for all $v\in \Omega_k$ then $(L/k,f)$ comes from $\Br k$ if and only if we can take $D$ to be principal.
\qed
\end{proposition}

\subsection{Cyclic Azumaya algebras}

Let $X$ be a smooth geometrically integral variety over a number field $k$. Assume that $X(k_v) \neq \emptyset$ for all $v \in \Omega_k$. By functoriality of the Brauer group we have maps $\Br k \to \Br X \to \Br k(X)$, where the first map is an injection  (see~\S\ref{S:BMobstructions}). 

Let $L/k$ be a cyclic extension. Define the set
\[
\Br_{\cyc}(X,L) := \genfrac{\{ }{\} }{0pt}{0}{\text{classes $[(L/k,f)]$ in the image of the}}{\text{map} \Br X/\Br k \to \Br k(X)/\Br k}
\]

\begin{lemma}
\label{L:H1 div is zero}
Viewing $\Delta := 1 - \sigma$ as an endomorphism of $\Div X_L$, we have $\ker N_{L/k} = \im \Delta$.
\end{lemma}

\begin{proof}
By Tate cohomology we know that $H^1(\Gal(L/k),\Div X_L) \cong \ker N_{L/k}/\im \Delta$.
On the other hand, this cohomology group is trivial: $\Div X_L$ is a permutation module, so the result follows from Shapiro's Lemma.
\end{proof}

The ideas behind the following theorem can be found in~\cite[\S4.3.2, esp.\ Lem. 4.18]{Bright}. 

\begin{theorem}
\label{T:cyclicalgebras}
Let $X$ be a $k$-variety as above. Let $H$ be an open normal subgroup of $G := \Gal(\kbar/k)$, such that $G/H$ is cyclic, generated by $\sigma$. Let $L$ be the fixed field of $H$. The map
\begin{equation*}
\psi\colon \ker \N/\im\Delta \to \Br_{\cyc}(X,L) \qquad
[D] \mapsto [(L/k,f)],
\end{equation*}
where $f \in k(X)^*$ is any function such that $N_{L/k}(D) = (f)$, is a group isomorphism.
\end{theorem}

\begin{proof}
First we check $\psi$ is well-defined by showing that
\begin{enumerate}
\item the class $[(L/k,f)]$ is independent of the choice of $f$: if $N_{L/k}(D) = (f) = (g)$, then $g = af$ for some $a \in k^*$. Since $(L/k,a)\in \Br k$, we obtain $[(L/k,f)] = [(L/k,g)]$.
\item if $D$ and $D'$ are linearly equivalent divisors in $\ker \N$, with $N_{L/k}(D) = (f)$ and $N_{L/k}(D') = (f')$, then $[(L/k,f)] = [(L/k,f')]$: equivalently, by Proposition~\ref{P:normsarecyclic} we need $(f/f')$ to be the norm of a principal divisor. Say $D = D' + (h)$. Then $(f/f') = N_{L/k}((h))$.
\item an element in $\im\Delta$ maps to zero: this is trivial.
\end{enumerate}

If $N_{L/k}(D) = (f)$ and $N_{L/k}(D') = (g)$ then
\begin{equation}
\label{E:homomorphism}
\psi([D] + [D']) = \psi([D + D']) = [(L/k,fg)] = [(L/k,f)] + [(L/k,g)] = \psi([D]) + \psi([D']),
\end{equation}
so $\psi$ is a homomorphism. The map $\psi$ is injective: if $\psi([D]) = [(L/k,f)]$ is $0$ in $\Br k(X)/\Br k$, then by Proposition~\ref{P:normsarecyclic} there exists an $h \in k(X_L)^*$ such that $(f) = N_{L/k}((h))$. Hence $D - (h) \in \ker N_{L/k} = \im \Delta$ (see Lemma~\ref{L:H1 div is zero}). Surjectivity also follows from Proposition~\ref{P:normsarecyclic}: given a class $[(L/k,f)]$, take any divisor $D$ such that $N_{L/k}(D) = (f)$; then $\psi([D]) = [(L/k,f)]$.
\end{proof}

\subsection{Cyclic algebras on rational surfaces}
\label{S:strategyondPs}
Let $X$ be a smooth, projective, geometrically integral rational surface over a number field $k$, and let $K$ be the splitting field of $X$. Assume that $X(\Aff_k) \neq \emptyset$. The inflation map
\begin{equation}
\label{E:inflation}
H^1(\Gal(K/k),\Pic X_K) \to H^1(\Gal(\kbar/k),\Pic \Xbar)
\end{equation}
is an isomorphism, because the cokernel maps into the first cohomology group of a free $\Z$-module with trivial action by a profinite group, so it is trivial. By~(\ref{E:HochshildSerreIso}) it follows that
\begin{equation}
\label{E:HSisom}
\Br X /\Br k \cong H^1(\Gal(K/k),\Pic X_K).
\end{equation}
Let $G = \Gal(K/k)$ and suppose that $H$ is a normal subgroup of $G$ such that $G/H$ is cyclic. Let $L$ be the fixed field of $H$. Since $X(\Aff_k) \neq \emptyset$, it follows from the Hochschild-Serre spectral sequence and~(\ref{E:HochshildSerreIso}) that
\begin{equation}
(\Pic X_K)^H \cong \Pic X_{L}.
\label{E:lower}
\end{equation}
We obtain an injection
\begin{equation}
\label{E:crucialinjection}
H^1(\Gal(L/k),\Pic X_{L}) \xrightarrow{\text{inf}} H^1(G,\Pic X_{K}) \cong \Br X /\Br k
\end{equation}
On the other hand, by Tate cohomology we know that
\[
H^1(\Gal(L/k),\Pic X_L) \cong \ker\N/\im \Delta.
\]
We can use Theorem~\ref{T:cyclicalgebras} to write down cyclic algebras $(L/k,f)$ in the image of the injection $\Br X/\Br k \to \Br k(X)/\Br k$. Since $G$ is finite, we may search through its subgroup lattice to find subgroups $H$ as above, and hence write down all the cyclic algebras in $\Br X/\Br k$.

It will be important for us to determine the functions $f$ above explicitly; to this end, we must make the isomorphism~(\ref{E:lower}) explicit. This is explained in the Appendix. 

\begin{remark}
Finding Brauer-Manin obstructions to the \emph{Hasse principle} on del Pezzo surfaces of degree greater than $1$ may require the injection~(\ref{E:crucialinjection}) to be an isomorphism. (This will be the case, for example, if $H^1(H,\Pic X_K) = 0$). We may need representative Azumaya algebras for \emph{every} class in $\Br X/\Br k$ to detect a Brauer-Manin obstruction (for example, see~\cite[9.4]{Corn2}). Obstructions to weak approximation only require one Azumaya algebra.
\end{remark}
\begin{remark}
Let $X$ be a diagonal del Pezzo surface of degree 1 over $\Q$ such that the order of $\Br X/\Br \Q$ is divisible by $3$. Let $K$ be the splitting field of $X$. Then an exhaustive computer search reveals that there does not exist a normal subgroup $H$ of $G := \Gal(K/\Q)$ such that $|G/H|$ is divisible by $3$. This means that any counterexamples to weak approximation over $\Q$ we find using the above strategy will always arise from $2$-torsion Azumaya algebras.
\end{remark}
\begin{remark}
Not all Brauer-Manin obstructions on del Pezzo surfaces arise from cyclic algebras: for example, see~\cite[Ex. 8]{KreschTschinkel}.
\end{remark}

\section{Exceptional curves on del Pezzo surfaces of degree $1$}\label{S:lines}

In this section we assume $k$ is algebraically closed. Let 
\[
\Gamma := V(z - Q(x,y), w - C(x,y)) \subseteq \PP_k(1,1,2,3),
\]
where $Q(x,y)$ and $C(x,y)$ are homogenous forms of degrees $2$ and $3$, respectively, in $k[x,y]$.
Define $\Gamma'$ as the image of $\Gamma$ under the Bertini involution (see \S\ref{SS:Bertini}). Note that $\Gamma \neq \Gamma'$.

\begin{lemma}
\label{L:Intersection3}
Let $X$ be a del Pezzo surface of degree $1$, given as a sextic hypersurface in $\PP_k(1,1,2,3)$.
If $\Gamma$ is a divisor on $X$ then so is $\Gamma'$; in this case $(\Gamma,\Gamma')_X = 3$.
\end{lemma}

\begin{proof}
It is clear that if $\Gamma$ is a divisor on $X$ then so is $\Gamma'$.  Assume first that $\Char k \neq 2$. Note $(\Gamma,\Gamma')_X$ is equal to the degree of the scheme $\Gamma\cap\Gamma'$, whose defining ideal is
\[
(z - Q(x,y),w - C(x,y),w + C(x,y)) = (z - Q(x,y),w,C(x,y)).
\]
We compute
\begin{align*}
\deg(\Proj k[x,y,z,w]/(z-Q,w,C)) &= \deg(\Proj k[x,y]/(C)) = 3.
\end{align*}
When $\Char k = 2$, the ideal of $\Gamma\cap\Gamma'$ is $(z + Q(x,y),w+ C(x,y),a)$. 
A calculation similar to the one above shows that $(\Gamma,\Gamma')_{X} = 3$.
\end{proof}

\subsection{The bianticanonical map}

Let $X$ be a del Pezzo surface of degree $1$ over $k$. The map
\[
\phi_2\colon X\to \PP(H^0(X,-2K_X)^*) = \PP^3_k.
\]
is known as the \defi{bianticanonical map}.  If $X = V(f(x,y,z,w)) \subseteq \PP_k(1,1,2,3)$, then the basis elements $x^2,xy,y^2,z$ for $H^0(X,-2K_X)$ are homogeneous coordinates for $\phi_2$ (see \S\ref{S:anticanonical models}).
Let $T_0,\dots,T_3$ be coordinates for $\PP^3_k$.  The map $\phi_2$ is $2$-to-$1$ onto the quadric cone $\mathcal{Q} = V(T_0T_2 - T_1^2)$. This cone is in turn isomorphic to the space $\PP_k(1,1,2)$ via the map
\begin{equation*}
j \colon \PP_k(1,1,2) \to  \mathcal{Q},\quad
[x:y:z] \mapsto [x^2:xy:y^2:z].
\end{equation*}
The composition $j^{-1}\circ \phi_2\colon X \to \PP_k(1,1,2)$ is just the restriction to $X$ of the natural projection $\PP_k(1,1,2,3) \dashrightarrow \PP_k(1,1,2)$. We fix the notation $\pi_2 := j^{-1}\circ \phi_2$ for future reference. 

\begin{lemma}[\cite{Cragnolini}]
\label{L:bianticanonical}
Let $V$ denote the vertex of the cone $\mathcal{Q}$, and let $\Gamma$ be an exceptional curve on $X$. Then $\phi_2|_\Gamma\colon \Gamma \to \phi_2(\Gamma)$ is $1$-to-$1$ and $\phi_2(\Gamma)$ is a smooth conic, the intersection of $\mathcal{Q}$ with a hyperplane $H$ that misses $V$.
\qed
\end{lemma}

\begin{remark}
The image of the anticanonical point under $\phi_2$ is $V\in \mathcal{Q}$.  By Lemma~\ref{L:bianticanonical}, the anticanonical point does not lie on any exceptional curve of $X$.
\label{R:anticanonicalpoint}
\end{remark}

\subsection{Proof of Theorem~\ref{T:lines}}

\begin{proof}[Proof of Theorem~\ref{T:lines}]
We may assume $k$ is algebraically closed, as the statement of the theorem is geometric. First, we show that any $\Gamma$ as in the theorem is an exceptional curve by proving that $(\Gamma,K_X)_X = (\Gamma,\Gamma)_X = -1$. Note $V(x) \in |{-K}_X|$. Hence
\[
(\Gamma,{-K}_X)_X = \deg(\Proj k[x,y,z,w]/(z - Q, w - C, x)) = \deg(\Proj k[y]) = 1.
\]
Let $D = V(z - Q(x,y))\subseteq \PP_k(1,1,2)$. Since $D$ is isomorphic under $j$ to a hyperplane section of the cone $\calQ$, we have $\pi_2^*(D) \in |{-2K_X}|$, so $(\pi_2^*(D),\pi_2^*(D))_X = 4$. Define $\Gamma'$ as the image of $\Gamma$ under the Bertini involution (see~\S\ref{SS:Bertini}). By Lemma~\ref{L:Intersection3}, $\Gamma'$ is a divisor on $X$. Since $\Gamma + \Gamma' \subseteq \pi_2^*(D)$, the divisor $\pi_2^*(D)$ is reducible, and since $\deg \pi_2 = 2$, it must consist of two distinct components with multiplicity $1$ (because $\Gamma \neq \Gamma'$), that is, $\Gamma + \Gamma' = \pi_2^*(D)$. The Bertini involution interchanges $\Gamma$ and $\Gamma'$, so we must have $(\Gamma, \Gamma + \Gamma')_X =  (\Gamma',\Gamma+\Gamma')_X = 2$. Thus $(\Gamma,\Gamma)_X = -1$ if and only if $(\Gamma,\Gamma')_X = 3$, but this follows from Lemma~\ref{L:Intersection3}. Hence $\Gamma$ is an exceptional curve. As above, we can show that $(\Gamma',-K_X)_X = 1$, so $\Gamma'$ is also an exceptional curve. \\

Now we prove the converse. Let $\Gamma$ be an exceptional curve on $X$.  By Lemma~\ref{L:bianticanonical} we know $\phi_2(\Gamma)$ is a smooth conic. It is isomorphic under the map $j$ to the curve $\pi_2(\Gamma)$ in $\PP_k(1,1,2)$. The equation for the conic in $\PP_k(1,1,2)$ can be written as $z = Q(x,y)$, where $Q(x,y)$ is homogenous of degree $2$ in $k[x,y]$ (the coefficient of $z$ is non-zero because $\phi_2(\Gamma)$ misses the vertex $V$ of the cone $\mathcal{Q}$). 

Let $D = V(z - Q(x,y))\subseteq \PP_k(1,1,2)$, as before. We have shown that $\Gamma \subseteq \pi_2^*(D)$.  Since $\pi_2^*(D) \in |{-2K_X}|$ as above, we have $(\pi_2^*(D),\Gamma)_X = 2$. If $\pi_2^*(D) = m\Gamma$ for some $m\geq 1$ then 
\[
2 = (\pi_2^*(D),\Gamma)_X = m(\Gamma,\Gamma)_X = -m,
\]
 a contradiction. Hence $\pi_2^*(D)$ is reducible, and $\pi_2^*(D) = \Gamma + \Gamma_1$ for some irreducible divisor $\Gamma_1\neq \Gamma$. Note that
\[
(\Gamma_1,\Gamma_1)_X = (\pi_2^*(D) - \Gamma,\pi_2^*(D) - \Gamma)_X = (-2K_X - \Gamma,-2K_X - \Gamma)_X= -1,
\]
and similarly $(\Gamma_1,{-K}_X)_X = 1$, so $\Gamma_1$ is an exceptional curve of $X$. We have
\[
\pi_2^*(D) = V(f(x,y,z,w), z - Q(x,y)).
\]
On the affine open subset where $x \neq 0$, the coordinate ring of $\pi_2^*(D)$ is 
\[
k[y,z,w]/(f(1,y,z,w),z - Q(1,y)) \cong k[y,w]/(f(1,y,Q(1,y),w)).
\]
Since $\pi_2^*(D)$ is reducible, the polynomial $f(1,y,Q(1,y),w)$ must factor, and degree considerations force a factorization of the following form:
\[
(w - C(1,y))(w - C'(1,y)),
\]
where $C(x,y)$ and $C'(x,y)$ are homogeneous forms of degree $3$. Hence $\Gamma$ has the form we claimed.
\end{proof}

\begin{remark}
The divisor $\Gamma_1$ in the proof above is the image of $\Gamma$ under the Bertini involution.
\end{remark}

\begin{remark}
We have used several ideas from the proof of~\cite[Key-lemma 2.7]{Cragnolini} to prove Theorem~\ref{T:lines}. The theorem can also be deduced from the work of Shioda on rational elliptic surfaces $S\to \PP^1$ (see~\cite[Thm.\ 10.10]{Shioda2}).  Shioda shows that rational elliptic surfaces have at most $240$ sections $\PP^1 \to S$ of a particular form, whose description bares a striking resemblance to the divisors of the form $\Gamma$ above.  A rational elliptic surface (over an algebraically closed field) with exactly $240$ of these special sections corresponds to the blow up of a del Pezzo surface $X$ of degree $1$ with center at the anticanonical point; the special sections of the elliptic surface are in one to one correspondence with the exceptional curves of $X$.  Under this correspondence, Shioda's explicit description of the $240$ sections becomes the explicit description of the exceptional curves of Theorem~\ref{T:lines}.  Cragnolini and Oliverio have a somewhat different description of the exceptional curves on a del Pezzo surface of degree $1$~\cite[Key-lemma 2.7]{Cragnolini} (see also~\cite[p. 68]{Demazure1980}).
\end{remark}

\begin{remark}
\label{R:unambiguous}
Suppose $k$ is not algebraically closed. The Bertini involution interchanges $\Gamma$ and $\Gamma'$; since it is defined over $k$ we conclude that
\[
{}^\sigma\! (\Gamma') = ({}^\sigma\Gamma)'\quad\text{for all } \sigma \in \Gal(\kbar/k).
\]
We will therefore use the unambiguous notation ${}^\sigma\Gamma'$ for this divisor.
\end{remark}

\section{Exceptional curves on diagonal surfaces}
\label{S:CurvesOnDiagonals}

We begin by studying the particular surface $Y$ given by the sextic $w^2 = z^3 + x^6 + y^6$ in $\PP_k(1,1,2,3)$. Suppose first that $k = \Qbar$. By Theorem~\ref{T:lines}, the exceptional curves on $Y$ are given as $V(w - C(x,y),z - Q(x,y))$, where
\[
C(x,y)^2 = Q(x,y)^3 + x^6 + y^6.
\]
Using Gr\"obner bases in {\tt Magma} to solve for the coefficients of $Q$ and $C$, we find $240$ exceptional curves, all defined over $\Q(\sqrt[3]{2},\zeta)$.

If $k$ is algebraically closed of characteristic $0$ the equations for the exceptional curves we calculated over $\Qbar$ give exceptional curves over $k$ via an embedding $\iota\colon\Qbar \hookrightarrow k$.

Now suppose $k$ is algebraically closed of characteristic $p > 3$. Let $W(k)$ be the ring of Witt vectors of $k$, and let $F(k)$ be its field of fractions. Let $\calX$ be the del Pezzo surface over $W(k)$ given by the equation $w^2 = z^3 + x^6 + y^6$ in $\PP_{W(k)}(1,1,2,3)$. The generic fiber of $\calX$ is a del Pezzo surface over $F(k)$. We may write down its $240$ exceptional curves as above: even though $F(k)$ is not algebraically closed, we may embed $\Q(\sqrt[3]{2},\zeta)$ in it, and this is enough to write down equations for all the exceptional curves.

The usual specialization map $\theta\colon \Pic\calX_{F(k)} \to \Pic \calX_k$ is a homomorphism (see~\cite[\S 20.3]{Fulton}). In other words, $\theta$ preserves the intersection pairings on $\Pic\calX_{F(k)}$ and $\Pic \calX_k$; it is injective because the pairing on $\Pic \calX_{F(k)}$ is nondegenerate. A standard computation shows that $\theta(K_{\calX_{F(k)}})=K_{\calX_k}$ (see~\cite[\S 20.3.1]{Fulton}). Hence $\theta$ maps exceptional curves to exceptional curves. The injectivity of $\theta$ then shows that the $240$ exceptional curves on $\calX_{F(k)}$ specialize to $240$ \emph{distinct} exceptional curves. \\

Let us drop the assumption that $k$ is algebraically closed. We turn to the general diagonal surface $X$ over $k$, given by $w^2 = z^3 + Ax^6 + By^6$. Fix a sixth root $\alpha$ of $A$
and a sixth root $\beta$ of $B$ in $\kbar$. If $\Gamma = V(z - Q(x,y), w - C(x,y))$ is an exceptional curve on $w^2 = z^3 + x^6 + y^6$, then $V(z - Q(\alpha x,\beta y), w - C(\alpha x,\beta y))$ is an exceptional curve on $X$, and vice versa. We deduce that the splitting field of $X$ is contained in $k(\zeta,\sqrt[3]{2},\alpha,\beta)$. 
 
\begin{proposition}
Let $k$ be a perfect field with $\Char p \neq 2,3$. Let $X$ be the del Pezzo surface of degree $1$ over $k$ given by 
\[
w^2 = z^3 + Ax^6 + By^6,
\]
in $\PP_{k}(1,1,2,3)$. Then the splitting field of $X$ is $K := k(\zeta,\sqrt[3]{2},\alpha,\beta)$. 
\label{P:Fieldofdefn}
\end{proposition}

\begin{proof}
Let $L$ denote the splitting field of $X$. The above discussion shows that $L\subseteq K$. Let $s = \sqrt[3]{2}$. By Theorem~\ref{T:lines}, the subschemes of $\PP_{k}(1,1,2,3)$ given by
\begin{align*}
V&(z - s\alpha\beta xy, w - \alpha^3 x^3 - \beta^3 y^3), \\
V&(z + s\zeta\alpha\beta xy, w - \alpha^3 x^3 - \beta^3 y^3), \\
V&(z + \alpha^2x^2 - s^2\zeta\beta^2y^2, w - s(\zeta + 1)\alpha^2\beta x^2y + (2\zeta - 1)\beta^3 y^3)\text{ and} \\
V&(z - s^2\zeta\alpha^2 x^2 + \beta^2 y^2, w - (2\zeta - 1)\alpha^3x^3 + s(\zeta + 1)\alpha\beta^2xy^2)
\end{align*}
 are exceptional curves on $X$. By definition of $L$, we find that
 \[
S :=  \{s\alpha\beta, s\zeta\alpha\beta, s(\zeta + 1)\alpha\beta^2, s(\zeta + 1)\alpha^2\beta\}\subseteq L.
\]
Taking the quotient of the second element of $S$ by the first shows that $\zeta \in L$. We also have $s(\zeta + 1)\alpha\beta \in L$, which shows $s(\zeta + 1)\alpha\beta^2/s(\zeta + 1)\alpha\beta = \beta \in L$. Similarly $s(\zeta + 1)\alpha^2\beta/s(\zeta + 1)\alpha\beta = \alpha \in L$. Finally, we deduce that $s\in L$. This shows $K \subseteq L$.
\end{proof}

To end our discussion on exceptional curves on diagonal surfaces, we give generators for $\Pic X_\kbar$ in terms of these curves. Consider the following exceptional curves on $X$:
{\scriptsize
     \begin{align*} \Gamma_{1} = V(z &+ \alpha^2x^2, w -     \beta^3y^3 ), \\
             \Gamma_{2} = V(z &- (-\zeta + 1)\alpha^2x^2, w + \beta^3y^3), \\
        \Gamma_{3} = V(z &-
        \zeta\alpha^2x^2 + s^2\beta^2y^2,  w -
        (s\zeta - 2s)\alpha^2\beta x^2y - (-2\zeta + 1)\beta^3y^3), \\
      \Gamma_{4} = V(z &+
        2\zeta\alpha^2x^2 - (2s\zeta - s)\alpha\beta xy - (-s^2\zeta + s^2)\beta^2y^2, \\ w &-
        3\alpha^3x^3 - (-2s\zeta - 2s)\alpha^2\beta x^2y - 3s^2\zeta\alpha\beta^2 xy^2 - (-2\zeta + 1)\beta^3y^3), \\
         \Gamma_{5} =  V(z &+
        2\zeta\alpha^2x^2 - (s\zeta - 2s)\alpha\beta xy - s^2\zeta\beta^2y^2 \\ w &+
        3\alpha^3x^3 - (4s\zeta - 2s)\alpha^2\beta x^2y - 3s^2\alpha\beta^2xy^2 - (-2\zeta + 1)\beta^3y^3), \\
                \Gamma_{6} = V(z &-
        (-s^2\zeta + s^2 - 2s + 2\zeta)\alpha^2x^2 - (2s^2\zeta - 2s^2 + 3s - 4\zeta)\alpha\beta xy - (-s^2\zeta + s^2 - 
            2s + 2\zeta)\beta^2y^2, \\ w &-
        (2s^2\zeta - 4s^2 + 2s\zeta + 2s - 6\zeta + 3)\alpha^3x^3 - (-5s^2\zeta + 10s^2 - 6s\zeta - 6s + 
            16\zeta - 8)\alpha^2\beta x^2y \\ &- (5s^2\zeta - 10s^2 + 6s\zeta + 6s - 16\zeta + 8)\alpha\beta^2 xy^2 - (-2s^2\zeta +
            4s^2 - 2s\zeta - 2s + 6\zeta - 3)\beta^3y^3), \\
       \Gamma_{7} = V(z &-
        (-s^2 - 2s\zeta + 2s + 2\zeta)\alpha^2x^2 - (-2s^2\zeta + 3s + 4\zeta - 4)\alpha\beta xy - (-s^2\zeta + s^2 + 2s\zeta - 2)\beta^2y^2, \\ w &-
        (2s^2\zeta + 2s^2 + 2s\zeta - 4s - 6\zeta + 3)\alpha^3x^3 - (10s^2\zeta - 5s^2 - 6s\zeta - 6s - 8\zeta
            + 16)\alpha^2\beta x^2y \\ &- (5s^2\zeta - 10s^2 - 12s\zeta + 6s + 8\zeta + 8)\alpha\beta xy^2 - (-2s^2\zeta - 
            2s^2 - 2s\zeta + 4s + 6\zeta - 3)\beta^3y^3), \\
           \Gamma_{8} = V(z &-
        (s^2\zeta + 2s\zeta + 2\zeta)\alpha^2x^2 - (2s^2 + 3s + 4)\alpha\beta xy - (-s^2\zeta + s^2 - 2s\zeta + 2s 
         - 2\zeta + 2)\beta^2y^2, \\ 
        w &- (-4s^2\zeta + 2s^2 - 4s\zeta + 2s - 6\zeta + 3)\alpha ^3x^3 - (-5s^2\zeta - 5s^2 - 6s\zeta 
- 6s - 8\zeta - 8)\alpha^2\beta x^2y \\ 
&- (5s^2\zeta - 10s^2 + 6s\zeta - 12s + 8\zeta - 16)\alpha\beta^2xy^2 - (4s^2\zeta - 2s^2 + 4s\zeta - 
2s + 6\zeta - 3)\beta^3y^3).
    \end{align*}}
A calculation shows that the above exceptional curves are all skew, that is, $(\Gamma_i,\Gamma_j)_X = 0$ for $i\neq j$. We will also need the exceptional curve
\[ 
\Gamma_{9} =  V(z - st\alpha\beta xy, w - \alpha^3x^3 + \beta^3y^3).
\]
The curve $\Gamma_9$ intersects $\Gamma_{1}$ and $\Gamma_{2}$ at exactly one point and is 
skew to all the other $\Gamma_i$.
\begin{proposition}
\label{P:generators}
Let $X$ be the del Pezzo surface over $k$ defined by 
\[
w^2 = z^3 + Ax^6 + By^6,
\]
in $\PP_k(1,1,2,3)$. Then $\Pic X_\kbar = \Pic X_K$ is the free abelian group with the classes of $\Gamma_i$ for $1\leq i \leq 8$ and $\Gamma_9 + \Gamma_1 + \Gamma_2$ as a basis.
\end{proposition}

\begin{proof}
By Proposition~\ref{P:Fieldofdefn} we know $K$ is the splitting field of $X$.  The classes of $\Gamma_i$ for $1\leq i \leq 8$ and $\Gamma_9 + \Gamma_1 + \Gamma_2$ generate a \emph{unimodular} sublattice of $\Pic X_K$ of rank $9$. Hence they span the whole lattice.
\end{proof}

\section{Galois action on $\Pic X_K$}
\label{S:galois}

Suppose $\sqrt[3]{2}, \zeta \notin k$ and let $X$ be a generic surface of  the form~(\ref{E:diagonalsurfaces}).  Let $K = k(\zeta,\sqrt[3]{2},\alpha,\beta)$, as above. The action of $\Gal(\kbar/k)$ on $\Pic X_\kbar$ factors through the finite quotient $\Gal(K/k)$, which acts on the coefficients of the equations defining generators of $\Pic X_K$ (cf.~\S\ref{S:GaloisAction}). The group $\Gal(K/k)$ has $4$ generators, which we will denote $\sigma, \tau, \iota_A, \iota_B$, whose action on the elements $\zeta, \sqrt[3]{2}, \alpha$ and $\beta$ is recorded in Table~\ref{Ta:Galois}. If $\sqrt[3]{2} \in k$ (resp. $\zeta \in k$), then we do not need the generator $\sigma$ (resp. $\tau$).

\begin{table}
\begin{center}
\begin{tabular}{|c|cccc|}
\hline
& $\sigma$ & $\tau$ & $\iota_A$ & $\iota_B$ \\
\hline
$\sqrt[3]{2}$ & $-\zeta\sqrt[3]{2}$ & $\sqrt[3]{2}$ & $\sqrt[3]{2}$ & $\sqrt[3]{2}$ \\
$\zeta$ & $\zeta$ & $\zeta^{-1}$ & $\zeta$ & $\zeta$ \\
$\alpha$ & $\alpha$ & $\alpha$ & $\zeta\alpha$ & $\alpha$ \\
$\beta$ & $\beta$ & $\beta $& $\beta$ & $\zeta\beta$ \\ 
\hline
\end{tabular}
\end{center}
\caption{Action of the generators of $\Gal(K/k)$, assuming $\sqrt[3]{2},\zeta \notin k$.}
\label{Ta:Galois}
\end{table}

Using the basis for $\Pic X_K$ of Proposition~\ref{P:generators} we can write $\sigma,\tau,\iota_A$ and $\iota_B$ as $9\times 9$ matrices with integer entries. This $9$-dimensional faithful representation is useful because the action of $\Gal(K/k)$ on $\Pic X_K$ becomes right matrix multiplication on the space of row vectors $\Z^9$. 

\begin{proof}[Proof of Theorem~\ref{T:possibleH1s}]
Assume first that $\sqrt[3]{2}, \zeta\notin k$. Then $G_0 := \langle \sigma,\tau,\iota_A,\iota_B \rangle \subseteq GL_9(\Z)$ is isomorphic to the generic image of $\Gal(\kbar/k)$ in $\Aut(\Pic X_\kbar)$ for a diagonal del Pezzo surface of degree 1. For a particular surface, a choice of sixth roots $\alpha$ and $\beta$ of $A$ and $B$, respectively, and a sixth root of unity $\zeta$ gives a realization of $G := \Gal(K/k)$ as a subgroup of $G_0$, where $K = k(\zeta,\sqrt[3]{2},\alpha,\beta)$. 

We turn this idea around by focusing on the subgroup lattice of $G_0$. We use {\tt Magma} to compute the first group cohomology (with coefficients in $\Pic X_K$) of subgroups in this lattice. We note there is no need to compute this cohomology group for every subgroup in the lattice. For example, any two subgroups of $G_0$ conjugate in $W(E_8)$ give rise to isomorphic cohomology groups. There are $448$ conjugacy classes of subgroups of $G_0$ in $W(E_8)$. 

We also note that  in order for a subgroup $G \subseteq G_0$ to correspond to at least one diagonal del Pezzo surface of degree 1, it is necessary that the natural map $G\to G_0/\langle \iota_A,\iota_B \rangle$ be surjective because $k(\zeta,\sqrt[3]{2}) \subseteq K$. This cuts the number of conjugacy classes for which we need to compute group cohomology to $242$. 

Fix a subgroup $G \subseteq G_0$. For each exceptional curve $\Gamma$ (given as a row vector in $\Z^9,$ using Proposition~\ref{P:generators}) we may compute the orbit of $\Gamma$ under the action of $G$. If there is a $G$-stable set of \emph{skew} exceptional curves, then any surface $X$ that has $G$ for its image of $\Gal(\kbar/k)$ in $\Aut(\Pic X_\kbar)$ is not minimal. Hence, we discard any such $G$.  This way we get rid of $58$ conjugacy classes of subgroups of $G_0$ and guarantee that surfaces we deal with in the rest of the paper are minimal.

The above reductions cut the number of candidate groups for $G$ to $184$. The results of our computations are summarized in Table~\ref{Ta:H1s}. For each abstract group $\Br X/\Br k$ we list the number $C(G)$ of conjugacy classes of subgroups of $G_0$ that give the listed cohomology group. We also give an example of a subgroup $G\subseteq G_0$ that has the given cohomology group, and a pair of elements $A,B \in k^*$ such that the surface $X$ of the form~(\ref{E:diagonalsurfaces}) realizes $G$ as a Galois group acting on $\Pic X_\kbar$. This shows all the possible cohomology groups \emph{do} occur.

If $\sqrt[3]{2} \in k$ yet $\zeta \notin k$ then we may repeat the above process starting with $G_0 = \langle \tau,\iota_A,\iota_B\rangle$. If $\zeta \in k$ yet $\sqrt[3]{2} \notin k$ then we use $G_0 = \langle \sigma,\iota_A,\iota_B\rangle$. Finally, if $\zeta, \sqrt[3]{2} \in k$ then we use $G_0 = \langle \iota_A,\iota_B\rangle$. The results in these three cases are summarized in Table~\ref{Ta:H1s}.
\end{proof}

\begin{table}
\begin{center}
\begin{tabular}{|c||c|c|c|c|l|}
\hline
& $\Br X/\Br k$ & $C(G)$ & Example of $G$ & $A, B$
& \multicolumn{1}{c|}{Restrictions}
\\
\hline\hline
$\sqrt[3]{2} \notin k$,\Strut
& $\{ 1\}$ & $65$ & $\genr{\sigma\iota_B^4, \tau,\iota_A^2}$ & $a^2c^6,\pm 4d^6$
& $a \notin \genr{2, k^{*3}}$
\\
$\zeta \notin k$
& $\Z/2\Z$  & $18$ & $\genr{\sigma\iota_A,\tau, \iota_B^3}$ & $4a^3c^6,b^3d^6$
& $a,b \notin \genr{2,-3, k^{*2}}$
\\
& $(\Z/2\Z)^2$  & $9$ & $\genr{\sigma,\tau, \iota_A^3\iota_B^3}$ & $a^3c^6,a^3d^6$
& $a \notin \genr{2,-3, k^{*2}}$
\\
& $(\Z/2\Z)^3$ & $4$ & $\genr{\sigma\iota_A^2, \tau, \iota_A^3\iota_B^3}$ & $16a^3c^6,a^3d^6$
& $a \notin \genr{2,-3, k^{*2}}$
\\
& $\Z/3\Z$ & $56$& $\genr{\sigma\iota_A\iota_B^2,\iota_A^3,\tau}$ & $4a^3c^6, \pm 16d^6$
& $a \notin \genr{-3,k^{*2}}$
\\ 
& $(\Z/3\Z)^2$ & $26$& $\genr{\tau, \sigma\iota_A^2\iota_B^2}$ & $ac^6,ad^6$
& $a \in \pm 16k^{*6}$
\\ 
& $\Z/6\Z$ & $6$& $\genr{\sigma\iota_A, \iota_A^3, \tau\iota_B}$
& $4a^3c^6,- 3d^6$
& $a \notin \genr{3, k^{*2}}$
\\ 
\hline
$\sqrt[3]{2}\in k$,\Strut
& $\{ 1\} $  & $11$ & $\genr{\tau,\iota_A\iota_B}$ & $ac^6,ad^6$
& $a \notin \genr{3,k^{*2},k^{*3}}$
\\
$\zeta \notin k$ & $\Z/2\Z$  & $7$ & $\genr{\tau,\iota_A\iota_B^3}$ & $ac^6,a^3d^6$
& $a \notin \genr{3,k^{*2},k^{*3}}$
\\
& $(\Z/2\Z)^2$  & $2$ & $\genr{\tau,\iota_A\iota_B^{5}}$ & $ac^6,a^5d^6$
& $a \notin \genr{3,k^{*2},k^{*3}}$
\\
& $(\Z/2\Z)^3$  & $1$ & $\genr{\tau,\iota_A^3,\iota_B^3}$ & $a^3c^6,b^3d^6$
& $a,b \notin \genr{-3,k^{*2}};\ a \neq b$
\\
& $(\Z/2\Z)^4$  & $2$ & $\genr{\tau,\iota_A^3\iota_B^3}$ & $a^3c^6,a^3d^6$
& $a \notin \genr{-3,k^{*2}}$
\\
& $\Z/3\Z$  & $8$& $\genr{\tau,\iota_A^2\iota_B^{5}}$ & $a^2c^6,a^5d^6$
& $a \notin \genr{3,k^{*2},k^{*3}}$
\\ 
& $(\Z/3\Z)^2$ & $5$& $\genr{\tau, \iota_A^2\iota_B^2}$ & $a^2c^6,a^2d^6$
& $a \notin \genr{3,k^{*3}}$
\\
& $\Z/6\Z$ & $4$& $\genr{\tau,\iota_A}$ & $ac^6,d^6$
& $a \notin \genr{3,k^{*2},k^{*3}}$
\\ 
\hline
$\sqrt[3]{2} \notin k$,\Strut
& $\{ 1\} $  & $26$ & $\genr{\sigma\iota_A^2\iota_B^2, \iota_A^3,\iota_B^3}$ & $16a^3c^6,16b^3d^6$
& $a,b \notin \genr{2, k^{*2}}; a \neq b$
\\
$\zeta\in k$
& $(\Z/2\Z)^2$ & $10$ & $\genr{\sigma\iota_B^{4}, \iota_A^3\iota_B^3}$ & $a^3c^6,4a^3d^6$
& $a \notin \genr{2, k^{*2}}$
\\
& $(\Z/2\Z)^4$  & $6$ & $\genr{\sigma, \iota_A^3\iota_B^3}$ & $a^3c^6,a^3d^6$
& $a \notin \genr{2, k^{*2}}$
\\
& $(\Z/2\Z)^6$  & $2$ & $\genr{\sigma\iota_B^2, \iota_A^3\iota_B^3}$ & $a^3c^6,16a^3d^6$
& $a \notin \genr{2, k^{*2}}$
\\
& $\Z/3\Z$  & $16$ & $\genr{\sigma\iota_A^2, \iota_A^{5}\iota_B^2}$ & $16a^5c^6,a^2d^6$
& $a \notin \genr{2, k^{*2}, k^{*3}}$
\\
& $(\Z/3\Z)^2$  & $16$& $\genr{\sigma\iota_A\iota_B^2}$ & $4a^3c^6,16d^6$
& $a \notin \genr{2, k^{*2}}$
\\ 
& $(\Z/3\Z)^3$ & $4$& $\genr{\sigma\iota_B^2, \iota_A^2\iota_B^2}$ & $a^2c^6,16a^2d^6$
& $a \notin \genr{2, k^{*3}}$
\\ 
& $(\Z/3\Z)^4$ & $3$& $\genr{\sigma\iota_A^2\iota_B^2}$ & $16c^6,16d^6$
& ---
\\ 
& $\Z/2\Z\times \Z/6\Z$ & $2$& $\genr{\sigma, \iota_A}$ & $a,d^6$
& $a \notin \genr{2, k^{*2}, k^{*3}}$
\\ 
& $(\Z/6\Z)^2$ & $2$& $\genr{\sigma\iota_B}$ & $c^6,4b^3d^6$
& $b \notin \genr{2, k^{*3}}$
\\ 
\hline
$\sqrt[3]{2} \in k$,\Strut
& $\{ 1\} $  & $5$ & $\genr{\iota_A\iota_B}$ & $ac^6,ad^6$
& $a \notin \genr{k^{*2},k^{*3}}$
\\
$\zeta \in k$
& $(\Z/2\Z)^2$  & $5$ & $\genr{\iota_A^3\iota_B}$ & $a^3c^6,ad^6$
& $a \notin \genr{k^{*2},k^{*3}}$
\\
& $(\Z/2\Z)^4$  & $1$ & $\genr{\iota_A\iota_B^{5}}$ & $ac^6,a^5d^6$
& $a \notin \genr{k^{*2},k^{*3}}$
\\
& $(\Z/2\Z)^6$  & $1$ & $\genr{\iota_A^3, \iota_B^3}$ & $a^3b^6,b^3d^6$
& $a,b \notin k^{*2};\ a \neq b$
\\
& $(\Z/2\Z)^8$  & $1$ & $\genr{\iota_A^3\iota_B^3}$ & $a^3c^6,a^3d^6$
& $a \notin k^{*2}$
\\
& $\Z/3\Z$  & $2$ & $\genr{\iota_A, \iota_B^2}$ & $ac^6,b^2d^6$
& $a \notin \genr{k^{*2},k^{*3}};\ b\notin k^{*3}$
\\
& $(\Z/3\Z)^2$  & $3$& $\genr{\iota_A^{5}\iota_B^2}$ & $a^5c^6,a^2d^6$
& $a \notin \genr{k^{*2},k^{*3}}$
\\ 
& $(\Z/3\Z)^4$ & $1$& $\genr{\iota_A^2\iota_B^2}$ & $a^2c^6,a^2d^6$
& $a \notin k^{*3}$
\\ 
& $(\Z/6\Z)^2$ & $2$& $\genr{\iota_B}$ & $c^6,bd^6$
& $a \notin \genr{k^{*2},k^{*3}}$
\\
\hline
\end{tabular}
\end{center}
\medskip
\caption{Possible groups $H^1(G,\Pic X)$. See the proof of Theorem~\ref{T:possibleH1s} for an explanation.}
\label{Ta:H1s}
\end{table}

Looking through our computations we observe that
\[
H^1(G_0,\Pic X_K) = 0,
\]
regardless of whether the elements $\sqrt[3]{2}$ and $\zeta$ belong to $k$ or not. This means that \emph{generically there is no Brauer--Manin obstruction to weak approximation} on diagonal del Pezzo surfaces of degree $1$ over a number field. 

\begin{remark}
In \cite[Theorem 4.1]{Corn2} Corn determines the possible groups 
\[
\Br X/\Br k \cong H^1(\Gal(\kbar/k),\Pic X_\kbar)
\] 
for all del Pezzo surfaces $X$ over a number field $k$. In particular, Corn shows the only primes that divide the order of this group are $2$, $3$ and $5$, and the latter can only occur when $X$ is of degree $1$. Unfortunately, \emph{diagonal} surfaces of degree $1$ cannot be used to give examples of $5$-torsion in $\Br X/\Br k$. This follows either from Theorem~\ref{T:possibleH1s} or, more easily, from the isomorphism~(\ref{E:inflation}): the group $H^1(\Gal(K/k),\Pic X_K)$ is annihilated by $[K:k]$, which divides $216$, by Proposition~\ref{P:Fieldofdefn}.
\end{remark}

\section{Counterexamples to Weak Approximation}
\label{S:counterexamples}

\subsection{A warm-up example} We begin with an example over $k = \Q(\zeta)$ for which we do not need to use the descent procedure described in the Appendix, and for $\Gal(K/k)$ is small. The presence of an obstruction to weak approximation on it cannot be explained by a conic bundle structure (see Remark~\ref{R:conic bundle structure}).

\begin{proposition}
\label{T:warmup}
Let $X$ be the del Pezzo surface of degree $1$ over $k = \Q(\zeta)$ given by
\[
w^2 = z^3 + 16x^6 + 16y^6
\] 
in $\PP_k(1,1,2,3)$.  Then $X$ is $k$-minimal and there is a Brauer-Manin obstruction to weak approximation on $X$. Moreover, the obstruction arises from a cyclic algebra class in $\Br X/\Br k$.
\end{proposition}

\begin{proof}
Let $\alpha = \beta = \sqrt[3]{4}$. By Proposition~\ref{P:Fieldofdefn}, the exceptional curves of $X$ are defined over $K := k(\sqrt[3]{2})$, and in the notation of \S\ref{S:galois} we have $G := \Gal(K/k) = \langle \rho\rangle$, where $\rho = \sigma\iota_A^2\iota_B^2$. Since $G$ is cyclic, we may apply the strategy of \S\ref{S:strategyondPs} by taking $H$ to be the trivial subgroup (so $L = K$). Using the basis for $\Pic X_K \cong \Z^9$ of Proposition~\ref{P:generators} we compute
\[
\ker\N/\im\Delta \cong (\Z/3\Z)^4;
\]
see Table~\ref{Ta:H1s}. The classes 
\begin{align*}
\mathfrak{h}_1 = [(0,  1,  0,  0,  0,  0,  0,  2, -1)], &\qquad \mathfrak{h}_2 = [(0,  0,  0,  0,  1,  0,  0,  2, -1)], \\
\mathfrak{h}_3 = [(0,  0,  0,  0,  0,  0,  1,  2, -1)], &\qquad \mathfrak{h}_4 = [(0,  0,  0,  0,  0,  0,  0,  3, -1)]
\end{align*}
of $\Pic X_K$ determine generators for this group.

Consider the divisor class $\mathfrak{h}_1 - \mathfrak{h}_2 = [\Gamma_2 - \Gamma_5] \in \Pic X_K$. By Theorem~\ref{T:cyclicalgebras}, this class gives a cyclic algebra $(K/k,f)$ in the image of the map $\Br X/\Br k \to \Br k(X)/\Br k$, where $f\in k(X)^*$ is any function such that $N_{K/k}(\Gamma_2 - \Gamma_5) = (f)$, that is, a function with zeroes along $\Gamma_2 + {}^{{\rho}}\Gamma_2 + {}^{{\rho}^2}\Gamma_2$ and poles along $\Gamma_5 + {}^{{\rho}}\Gamma_5 + {}^{{\rho}^2}\Gamma_5$.  
Using the explicit equations for $\Gamma_2$ in \S\ref{S:CurvesOnDiagonals} we see that the polynomial $w + 4y^3$ vanishes along $\Gamma_2 + {}^{{\rho}}\Gamma_2 + {}^{{\rho}^2}\Gamma_2$. 

Let $I$ be the ideal of functions that vanish on $\Gamma_5, {}^{{\rho}}\Gamma_5$ and ${}^{{\rho}^2}\Gamma_5$. Explicitly,
\[
I = (z - Q_5,w - C_5) \cap (z - {}^\rho Q_5, w - {}^\rho C_5) \cap (z - {}^{\rho^2}\!Q_5, w - {}^{\rho^2}\!C_5),
\]
where $Q_5$ and $C_5$ are the quadratic and cubic forms, respectively, corresponding to $\Gamma_5$, and, for example, ${}^\rho Q_5$ is the result of applying $\rho$ to the coefficients of $Q_5$. We compute a Gr\"obner basis for $I$ (under the lexicographic order $w>z>y>x$) and find the polynomial $w + (2\zeta + 2)zy + (-8\zeta + 4)y^3 + 12x^3$ in this basis. Hence
\[
f := \frac{w + 4y^3}{w + (2\zeta + 2)zy + (-8\zeta + 4)y^3 + 12x^3}
\]
has the required zeroes and poles.

Consider the following rational points of $X$:
\[
P_{1} = [1:0:0:4] \quad\text{and}\quad P_{2} = [0:1:0:4].
\]
Let $\sA$ be the Azumaya algebra of $X$ corresponding to $(K/k,f)$. Specializing the algebra $\sA$ at $P_{1}$ we obtain the cyclic algebra $\sA(P_1) = (K/k,1/4)$ over $k$. On the other hand, specializing at $P_2$ we compute $\sA(P_2) = (K/k,1/(1 - \zeta)) = (K/k,\zeta)$.

Let $\pp$ be the unique prime above $3$ in $k$. To compute the invariants we observe that
\[
\inv_\pp(\sA(P_i)) = \frac{1}{3}[f(P_i),2]_\pp \in \Q/\Z,
\]
where $[f(P_i)_\pp,2]_\pp \in \Z/3\Z$ is the (additive) norm residue symbol. We compute $[1/4,2]_\pp \equiv 0 \bmod 3$ (using~\cite[(77)]{CTKS87}) and $[\zeta,2]_\pp \equiv 1 \bmod 3$ (using biadditivity of the norm residue symbol and~\cite[(75)]{CTKS87} with $\theta = -\zeta$, $a = 1$). Let $P \in X(\Aff_k)$ be the point that is equal to $P_{1}$ at all places except $\pp$, and is $P_{2}$ at $\pp$. Then 
\[
\sum_{v} \inv_v(\sA(P_v)) = 1/3,
\]
so $P \in X(\Aff_k)\setminus X(\Aff_k)^{\Br}$ and $X$ is a counterexample to weak approximation. 

To see that $X$ is $k$-minimal, see the proof of Theorem~\ref{T:possibleH1s}: the surface $X$ appears as the example in the twelfth line from the bottom of Table~\ref{Ta:H1s}.
\end{proof}

\begin{remark}
\label{R:conic bundle structure}
The surface $X$ of Proposition~\ref{T:warmup} is not birational to a conic bundle $C$, since the birational invariant $\Br X/\Br k$ is isomorphic to $(\Z/3\Z)^4$, while $\Br C/\Br k$ is always $2$-torsion. In particular, the failure of weak approximation cannot be accounted for by the presence of a conic bundle structure.
\end{remark}

\subsection{Main Theorem}
We are ready to prove our main theorem.

\begin{proof}[Proof of Theorem~\ref{T:maintheorem}]
Let $\alpha = \beta = \sqrt{p}$. By Proposition~\ref{P:Fieldofdefn}, the exceptional curves of $X$ are defined over $K := \Q(\zeta,\sqrt[3]{2},\sqrt{p})$, and in the notation of \S\ref{S:galois} we have $G := \Gal(K/\Q) = \langle \sigma,\tau,\iota_A^3\iota_B^3\rangle$. One easily checks that the element $\rho := \iota_A^3\iota_B^3$ acts on exceptional curves as the Bertini involution of the surface (see \S\ref{SS:Bertini}).

The subgroup $H := \langle\sigma,\tau\rangle$ of $G$ has index $2$; hence it is normal and $G/H$ is cyclic. Thus, we are in the situation described in \S\ref{S:strategyondPs}, that is,
\[
H^1(\Gal(L/\Q),\Pic X_{L}) \hookrightarrow \Br X /\Br \Q,
\]
where $L = K^H$ is $\Q(\sqrt{p})$ in this case. The injection is in fact an isomorphism because $H^1(H,\Pic X_K) = 0$, though we will not use this fact. Using the basis for $\Pic X_K \cong \Z^9$ of Proposition~\ref{P:generators} we compute
\[
\ker\N/\im\Delta \cong \Z/2\Z\times \Z/2\Z.
\]
The classes 
\begin{equation}
\label{E:generatorsForPic}
\mathfrak{h}_1 = [(2,1,1,1,1,0,1,2,-3)]\qquad\text{and}\qquad \mathfrak{h}_2 = [(0,0,0,0,0,1,0,-1,0)]
\end{equation}
of $\Pic X_L$ generate this group.

Next, we apply the procedure of the Appendix to descend the line bundle $\OO_{X_K}(\Gamma_6 - \Gamma_8)$ in the class of $\mathfrak{h}_2$ to a line bundle defined over $\Q(\sqrt{p})$. We must give isomorphisms
\[
f_h: \OO_{X_K}(\Gamma_6 - \Gamma_8) \to \OO_{X_K}(\null^h\Gamma_6 - \null^h\Gamma_8),
\]
one for each $h\in H$, satisfying the cocycle condition.  In this case $H$ is isomorphic to the symmetric group on $3$ elements, with presentation
\[
H = \langle\sigma,\tau \,|\, \sigma^3 = \tau^2 = 1, \sigma\tau = \tau\sigma^2\rangle,
\]
so it is enough to find isomorphisms $f_\sigma$ and $f_\tau$ as above such that
\begin{align*}
\null^{\sigma^2}\!f_\sigma\circ\null^\sigma\!f_\sigma \circ f_\sigma & = id, \\
\null^\tau\!f_\tau\circ f_\tau &= id, \\
\null^\sigma\!f_\tau \circ f_\sigma &= \null^{\tau\sigma}\!f_\sigma\circ\null^\tau\!f_\sigma\circ f_\tau.
\end{align*}
For example, the map $f_\sigma$ is just multiplication by a function having zeroes at $\Gamma_6$ and $\null^\sigma \Gamma_8$ and poles at $\Gamma_8$ and $\null^\sigma\Gamma_6$. We also denote  this function $f_\sigma$, and find it as follows. First, take a function that vanishes on $\Gamma_8$, $\null^\sigma\Gamma_6$, and possibly some extra lines. For example, recall that
\begin{align*}
\Gamma_6 = V(z - Q_6(x,y), w - C_6(x,y)), \quad\text{and}\quad
\Gamma_8 = V(z - Q_8(x,y), w - C_8(x,y)),
\end{align*}
where $Q_6$ and $Q_8$ (resp $C_6$ and $C_8$) are the quadratic (resp.\ cubic) forms in $x$ and $y$, corresponding to $\Gamma_6$ and $\Gamma_8$ given in \S\ref{S:CurvesOnDiagonals}. Let $\null^\sigma\!Q_6$ denote the result of applying $\sigma$ to the coefficients of $Q_6$, and similarly for the other binary and cubic forms. The function
\[
g_1 = (z - \null^\sigma\!Q_6(x,y))(z -  Q_8(x,y))
\]
vanishes on the exceptional curves\footnote{The notation $\null^\sigma\Gamma_6'$ is unambiguous (cf. Remark~\ref{R:unambiguous}).} $\null^\sigma\Gamma_6$, $\null^\sigma\Gamma_6'$, $\Gamma_8$ and $\Gamma_8'$. Let $I$ be the ideal of functions that vanish on $\Gamma_6$, $\null^\sigma\Gamma_8$, $\null^\sigma\Gamma_6'$ and $\Gamma_8'$. Explicitly,
\[
I = (z - Q_6, w - C_6) \cap (z - {}^\sigma\!Q_8, w - {}^\sigma\!C_8) \cap 
(z - {}^\sigma\!Q_6, w + {}^\sigma\!C_6) \cap (z - Q_8, w + C_8).
\]
We compute a Gr\"obner basis for $I$ (under the lexicographic order $w>z>y>x$) and find the following degree $4$ polynomial in the basis:
{\scriptsize
\begin{align*}
f_1 &= 6\sqrt{p}wy + 3\sqrt{p}(\zeta - 1)(s^2 + 2)wx + (-2\zeta
+ 1)sz^2 + 2p(2\zeta - 1)(s^2 + s + 1)zy^2 \\
&\quad + p(-\zeta - 1)(3s^2 + 2s + 2)zyx + 2p(-\zeta + 2)(s^2 + s + 1)zx^2 + 2p^2(2\zeta - 1)(s^2 + 1)y^4 \\
&\quad  + p^2(-\zeta - 1)(3s^2 + 2s +2)y^3x + 2p^2(-\zeta + 2)(s^2 + s + 1)y^2x^2  + 2p^2(2\zeta
- 1)(s + 1)yx^3\\
&\quad  + p^2(\zeta + 1)(s^2 - 2)x^4.
\end{align*}
}
The function $f_1/g_1$ has the right zeroes and poles to be $f_\sigma$. We set
\[
f_\sigma := \frac{1}{(-2\zeta + 1)s}\cdot\frac{f_1}{g_1}.
\]
The constant in front of $f_1/g_1$ is a normalization factor, making $f_\sigma([0:0:1:1]) = 1$. 

Similarly, $f_\tau$ denotes a function with zeroes at $\Gamma_6$ and $\null^\tau\Gamma_8$ and poles at $\Gamma_8$ and $\null^\tau\Gamma_6$. Let
{\scriptsize
\begin{align*}
g_2 &= (z - \null^\tau\!Q_6(x,y))(z - Q_8(x,y)), \\
f_2 &= 6\sqrt{p}wy - \sqrt{p}wx + (-2\zeta + 1)kz^2 + 2p(2\zeta - 1)(k^2 + k +1)zy^2 + 2p(2\zeta - 1)(k + 1)zyx \\
    &\quad + 2p(2\zeta - 1)(k^2 + k + 1)zx^2 + 2p^2(2\zeta -
    1)(k^2 + 1)y^4 + 2p^2(2\zeta - 1)(k + 1)y^3x \\
    &\quad + 2p^2(2\zeta - 1)(k^2 + k + 1)y^2x^2 + 
    2p^2(2\zeta - 1)(k + 1)yx^3 + 2p^2(2\zeta - 1)(k^2 + 1)x^4.
\end{align*}
}
Then the function
\[
f_\tau := \frac{1}{(-2\zeta + 1)k}\cdot\frac{f_2}{g_2}
\]
has zeroes at $\Gamma_6$ and $\null^\tau\Gamma_8$ and poles at $\Gamma_8$ and $\null^\tau\Gamma_6$. Because of the normalization, $f_\tau$ and $f_\sigma$ satisfy the cocycle condition. Thus $\OO_{X_K}(\Gamma_6 - \Gamma_8)$ descends to a line bundle $\sF$ over $L$, as we expected. It remains to find a divisor over $L$ in the class of $\sF$. To this end, we average the rational section $1$ of $\OO_{X_K}(\Gamma_6 - \Gamma_8)$ over the group $H$ to obtain a rational section 
\[
\s = \sum_{h \in H}{}^{h^{-1}}\!(f_h) = 1 + {}^{\sigma^2}\!f_\sigma + 
{}^\tau\!f_\tau + {}^{\sigma^2}\!f_\sigma\cdot {}^{\sigma}\!f_\sigma +
{}^{\sigma^2}\!f_\sigma\cdot{}^{\sigma}\!f_\sigma\cdot{}^{\sigma\tau}\!f_\tau +
{}^{\sigma^2}\!f_\sigma\cdot {}^{\tau\sigma}\!f_\tau
\]
of $\sF$. The common denominator of $\mathfrak{s}$ is
\[
{}^{\sigma^2}\!g_1\cdot {}^\tau\!g_2 \cdot {}^\sigma\!g_1 \cdot {}^{\sigma\tau}\!g_2 \cdot {}^{\tau\sigma}\!g_2 .
\]
By definition of $g_1$ and $g_2$ this denominator vanishes along the divisor
 \begin{align*}
 &2\Gamma_6 + 2\Gamma_6' + {}^{\sigma^2}\Gamma_8 + {}^{\sigma^2}\Gamma_8' + {}^{\tau}\Gamma_8 + {}^{\tau}\Gamma_8' + 2\big({}^{\sigma^2}\Gamma_6\big) + 2\big({}^{\sigma^2}\Gamma_6'\big) \\
&\quad + 
{}^{\sigma}\Gamma_8 +\null^{\sigma}\Gamma_8' + \null^{\sigma\tau}\Gamma_8 + \null^{\sigma\tau}\Gamma_8' + \null^{\sigma}\Gamma_6 + \null^{\sigma}\Gamma_6' + \null^{\tau\sigma}\Gamma_8 + \null^{\tau\sigma}\Gamma_8'.
 \end{align*}
 Here $2\Gamma_6$ means, for example, that the denominator vanishes on this curve to order $2$. The numerator of $\mathfrak{s}$ vanishes along the divisor
\[
 \Gamma_6 + 2\Gamma_6' + {}^{\sigma^2}\Gamma_8' + {}^\tau\Gamma_8' + 2\big({}^{\sigma^2}\Gamma_6\big) + 2\big({}^{\sigma^2}\Gamma_6'\big) + {}^\sigma\Gamma_8' + {}^{\sigma\tau}\Gamma_8' + {}^\sigma\Gamma_6 + {}^\sigma\Gamma_6' + {}^{\tau\sigma}\Gamma_8' + Z,
\]
 where $Z$ is some curve on $X$.  Thus, as a rational function, $\s$ has a zero of order $1$ along $Z$ and poles of order $1$ along the divisor
\[
 P := \Gamma_6 + {}^{\sigma^2}\Gamma_8 +  {}^\tau\Gamma_8 + {}^{\sigma\tau}\Gamma_8 + {}^\sigma\Gamma_8 + \null^{\tau\sigma}\Gamma_8
\]
As a \emph{rational section} of $\OO_{X_K}(\Gamma_6 - \Gamma_8)$, $\s$ has a zero along $Z$ and a pole along
\[
P' := \Gamma_8 + {}^{\sigma^2}\Gamma_8 +  {}^\tau\Gamma_8 + {}^{\sigma\tau}\Gamma_8 + {}^\sigma\Gamma_8 + \null^{\tau\sigma}\Gamma_8 = \sum_{h \in H} {}^h\Gamma_8.
\]
The divisor $Z - P' \in \Div X_L$ represents the class of $\Gamma_6 - \Gamma_8$. 

Let $\bar{\rho}$ be the class of $\rho$ in $\Gal(L/\Q)$. By Theorem~\ref{T:cyclicalgebras}, the class $[Z - P']$ gives a cyclic algebra $(\Q(\sqrt{p})/\Q,f)$ in $\Br X/\Br \Q$, where $f\in \Q(X)^*$ is any function such that 
\[
N_{\Q(\sqrt{p})/\Q}(Z - P') = Z + {}^{\bar{\rho}}\!Z - (P' + {}^{\bar{\rho}}\!P')= (f).
\]

We find an explicit $f$. The numerator of $\s$ (after cancelling out common divisors) is a polynomial of degree $12$ in $\Q(\sqrt[3]{2},\zeta,\sqrt{p})[x,y,z,w]$. We may express it as
\[
p_1 + \sqrt[3]{2}p_2 + \sqrt[3]{4}p_3 + \zeta p_4 + \zeta\sqrt[3]{2}p_5 + \zeta\sqrt[3]{4}p_6,
\]
where $p_i \in \Q(\sqrt{p})[x,y,z,w]$ for $i = 1,\dots, 6$. Then $Z = V(p_1,\dots,p_6)$.
We find constants $b_i \in \Q(\sqrt{p})$ such that the polynomial $q = \sum_i b_ip_i$ belongs to $\Q[x,y,z,w]$; then the polynomial $q$ vanishes on $Z \cup {}^{\bar{\rho}}\!Z$ and is a suitable numerator for $f$. A little linear algebra reveals that
{\scriptsize
\begin{align*}
q &= 12z^6 -72pz^5y^2 -192pz^5yx -48pz^5x^2 +
300p^2z^4y^4 +
    600p^2z^4y^3x + 576p^2z^4y^2x^2 + 408p^2z^4yx^3 \\
    &\quad + 156p^2z^4x^4 -288p^3z^3y^6 -720p^3z^3y^5x -888p^3z^3y^4x^2 
    -768p^3z^3y^3x^3 -756p^3z^3y^2x^4 -264p^3z^3yx^5 \\
    &\quad  -204p^3z^3x^6 + 144p^4z^2y^8 + 456p^4z^2y^7x 
    + 1032p^4z^2y^6x^2 + 1080p^4z^2y^5x^3 + 756p^4z^2y^4x^4 + 864p^4z^2y^3x^5 \\
    &\quad  + 684p^4z^2y^2x^6 + 456p^4z^2yx^7 -48p^4z^2x^8 + 192p^5zy^{10} -48p^5zy^9x 
    -720p^5zy^8x^2 -1104p^5zy^7x^3\\
    &\quad  -600p^5zy^6x^4 
    -216p^5zy^5x^5 -240p^5zy^4x^6 -480p^5zy^3x^7 
    -504p^5zy^2x^8 -24p^5zyx^9 + 48p^5zx^{10} \\
    &\quad -192p^6y^{12} -288p^6y^{11}x + 192p^6y^{10}x^2 + 528p^6y^9x^3 + 432p^6y^8x^4 +
    168p^6y^7x^5 -192p^6y^6x^6  -288p^6y^5x^7\\
    &\quad + 192p^6y^4x^8 +
    312p^6y^3x^9 -48p^6yx^{11}.
\end{align*}
}
Now we look for a polynomial $r$ of the same degree as $q$ vanishing on $P' + {}^{\bar{\rho}}\!P'$. Since $\rho$ acts as the Bertini involution $\Gamma \mapsto \Gamma'$ on exceptional curves, we have
\[
P' + {}^{\bar{\rho}}\!P' = \sum_{h \in H} {}^h(\Gamma_8 + \Gamma_8').
\]
The polynomial $z - Q_8(x,y)$ vanishes on $\Gamma_8 + \Gamma_8'$. Hence we may take $r = \prod_{h \in H}(z - {}^hQ_8(x,y)),$ and obtain
{\scriptsize
\begin{align*}
r &= z^6 - 6pz^5y^2 - 24pz^5yx - 6pz^5x^2 + 36p^2z^4y^4 + 78p^2z^4y^3x +
132p^2z^4y^2x^2 + 78p^2z^4yx^3 + 36p^2z^4x^4 \\
  & \quad  + 8p^3z^3y^6 - 60p^3z^3y^5x - 168p^3z^3y^4x^2 -
276p^3z^3y^3x^3
    - 168p^3z^3y^2x^4 - 60p^3z^3yx^5 + 8p^3z^3x^6  - 24p^4z^2y^8 \\
    &\quad - 24p^4z^2y^7x +
    156p^4z^2y^6x^2 + 396p^4z^2y^5x^3 + 540p^4z^2y^4x^4 + 396p^4z^2y^3x^5 + 156p^4z^2y^2x^6 - 24p^4z^2yx^7 \\
     &\quad - 24p^4z^2x^8 + 24p^5zy^9x + 24p^5zy^8x^2 - 120p^5zy^7x^3 -
    324p^5zy^6x^4 - 432p^5zy^5x^5 - 324p^5zy^4x^6 \\
    &\quad  - 120p^5zy^3x^7 + 24p^5zy^2x^8 +
    24p^5zyx^9 + 16p^6y^{12} + 48p^6y^{11}x + 48p^6y^{10}x^2 + 48p^6y^9x^3 + 120p^6y^8x^4\\
    &\quad   +  192p^6y^7x^5 + 212p^6y^6x^6 + 192p^6y^5x^7 + 120p^6y^4x^8 + 48p^6y^3x^9 +
    48p^6y^2x^{10} + 48p^6yx^{11} + 16p^6x^{12}.
\end{align*}
}
Let $f = q/r$ and let $\sA$ denote the Azumaya algebra on $X$ corresponding to $(L/\Q,f)$. There are two obvious rational points on the surface $X$ other than the anticanonical point, namely,
\[
P_{1} = [1:0:-p:0] \quad\text{and}\quad P_{2} = [0:1:-p:0].
\]
Specializing the algebra $\sA$ at $P_{1}$ we obtain the quaternion algebra $(p,12) \cong (p,3)$ over $\Q$. The invariant of this algebra at a prime $q$ is readily calculated using the Hilbert symbol $[\,\cdot\,  ,\cdot\, ]_q \in \{\pm 1\}$ of the quaternion algebra (cf.~\cite{serre}) via the formula
\[
\inv_q (a,b) = \frac{1 - [a,b]_q}{4} \in \Q/\Z.
\]
Using the formulas for the Hilbert symbol in~\cite{serre}, we find that
\[
[p,3]_q = 
\begin{cases}
(-1)^{(p-1)/2} & \text{if } q = 2, \\
p\overwithdelims () 3& \text{if } q = 3, \\
3\overwithdelims () p & \text{if } q = p, \\
1 & \text{otherwise,} \\
\end{cases}
\]
where $p\overwithdelims () q$ is the usual Legendre symbol. On the other hand, specializing $\sA$ at $P_{2}$ we obtain the quaternion algebra $(p,16) \cong (p,1)$ over $\Q$. We find that $[p,1]_q = 1$ for all primes $q$. Hence
\begin{equation}
\label{E:congruences}
\inv_3 (p,3) \neq \inv_3 (p,1) \text{ if } p \equiv 5\bmod 6
\quad \text{and}\quad \inv_2 (p,3) \neq \inv_2 (p,1) \text{ if } p \equiv 3\bmod 4.
\end{equation}
Let $P \in X(\Aff_\Q)$ be the point that is equal to $P_{1}$ at all places except $p$, and is $P_{2}$ at $p$. Then by~(\ref{E:congruences}) it follows that if $p\equiv 5 \bmod 6$ then
\[
\sum_{v} \inv_v(\sA(P_v)) = 1/2.
\]
Similarly, if $P' \in X(\Aff_\Q)$ is the point that is equal to $P_{1}$ at all places except $2$, and is $P_{2}$ at $2$, then by~(\ref{E:congruences}) we find that the sum of invariants is again $1/2$ when  $p\equiv 3 \bmod 4$.

In either case, we have shown that if $p\not\equiv 1 \bmod 12$ then $X(\Aff_\Q) \neq X(\Aff_\Q)^{\Br}$, and hence $X$ does not satisfy weak approximation.

Finally, we note that $\Pic X = (\Pic X_L)^{\Gal(L/\Q)} = \ker\Delta = \Z$, generated by the anticanonical class. In fact, $\Pic X_L \cong \Z^3$, generated by the classes~(\ref{E:generatorsForPic}) and the anticanonical class, and $\rho$ acts nontrivially on the classes~(\ref{E:generatorsForPic}).  Hence $X$ is minimal.
\end{proof}

\appendix

\section{Galois Descent of Line Bundles}
\label{S:GaloisDescent}

To make the isomorphism~(\ref{E:lower}) explicit we need the theory of Galois descent of line bundles, which is a special case of the theory of descent of quasi-coherent sheaves over faithfully flat and quasi-compact morphisms. Good references for Galois descent are \cite{BLR} and \cite{KT3}. For the general theory of descent see \cite{SGA1}. 

Let $K/k$ be a finite Galois extension of number fields. For every element $\sigma \in \Gal(K/k)$ let $\tilde{\sigma} : \Spec K \to \Spec K$ denote the corresponding morphism. Let $X$ be a $k$-scheme, and suppose we are given a line bundle $\widetilde{\sF}$ on the $K$-scheme $X_K$, together with a collection of isomorphisms $f_{\sigma}: \widetilde{\sF} \to \tilde{\sigma}^*\widetilde{\sF}$ such that
\begin{equation}
\label{E:cocyclecondition}
f_{\sigma\tau} = \null^\sigma\!f_\tau\circ f_\sigma \qquad\text{for all }\sigma,\tau \in \Gal(K/k),
\end{equation}
where $\null^\sigma\!f_\tau := \tilde{\sigma}^*f_\tau$. Then there exists a sheaf $\sF$ on $X$, and an isomorphism $\lambda\colon\sF_K \to \widetilde{\sF}$ such that $f_\sigma = \null^\sigma\!\lambda\circ\lambda^{-1}$ for all $\sigma$. Together, the equalities~(\ref{E:cocyclecondition}) are referred to as the \defi{cocycle condition}. 

If $X$ is a geometrically integral $k$-scheme, then $\widetilde{\sF} = \OO_{X_K}(D)$ for some divisor $D\in \Div X_K$, and $f_\sigma$ can be regarded as a function (up to multiplication by a scalar) whose associated divisor is $D - \null^\sigma\! D$. If $X(K) \neq \emptyset$ then one may use a point in $P \in X(K)$ to normalize the functions so that $f_\sigma$ acts as the identity in the fiber of $\widetilde{\sF}$ at $P$. We usually don't know if $X(K)$ is empty or not, but in the case of del Pezzo surfaces of degree $1$ over $k$ we have the anticanonical point.

To obtain a divisor for the descended line bundle, we take a rational section $\xi$ of $\widetilde{\sF}$ and we ``average it'' over the Galois group $G$ to obtain a rational section of $\sF$
\[
\s := \sum_{\sigma\in G} \null^{\sigma^{-1}}\!(f_\sigma(\xi)).
\]
Note it may be necessary to change the choice of $\xi$ to make $\s$ nonzero.
The divisor of zeroes of $\s$, with respect to local trivializations for $\sF$, gives a line bundle isomorphic to the descended line bundle. We often use the rational section $\xi = 1$, and since $f_\sigma$ acts by multiplication, we obtain
$ \s = \sum_{\sigma\in G} \null^{\sigma^{-1}}\!(f_\sigma)$
in this case.


\end{document}